\documentclass[12pt]{amsart}
\usepackage{amssymb, amscd, amsmath}

\theoremstyle{plain}
\newtheorem{thm}{Theorem}[section]

\newtheorem{conj}[thm]{Conjecture}
\newtheorem{cor}[thm]{Corollary}

\newtheorem{exam}[thm]{Example}

\newtheorem{propn}[thm]{Proposition}

\newtheorem{rem}[thm]{Remark}

\def\C{{\mathbb C}}                               
\def\F{{\mathbb F}}                               
\def\M{{\mathbb M}}                               
\def\N{{\mathbb N}}                               
\def\Q{{\mathbb Q}}                               
\def\R{{\mathbb R}}                               
\def\Z{{\mathbb Z}}                               

\def\Aut{\operatorname{Aut}}                      
\def\End{{\operatorname{End}}}                    
\def\GL{\operatorname{GL}}                        
\def\Gal{\operatorname{Gal}}                      
\def\Hom{\operatorname{Hom}}                      
\def\Ind{\operatorname{Ind}}
\def\Nm{\operatorname{Nm}}
\def\SL{\operatorname{SL}}                        
\def\SO{\operatorname{SO}}                        
\def\SU{\operatorname{SU}}          
\def\U{\operatorname{U}}          
\def\Sp{\operatorname{Sp}}                        
\def\Oc{\mathbb O}

\def\det{\operatorname{det}}                      
\def\tr{\operatorname{tr}}                        

\numberwithin{equation}{section}

\newtheorem{Lem}[equation]{Lemma}
\newtheorem{Prop}[equation]{Proposition}

\begin{document}

\title {Notes on representations of finite groups of Lie type}

\author{Dipendra Prasad}

\address{School of Mathematics, Tata Institute of Fundamental
Research, Colaba, Mumbai-400005, INDIA}
\email{dprasad@math.tifr.res.in}
\maketitle

\vskip10mm
\begin{quote}{\sf \footnotesize
These are the notes of some lectures given by the author for a workshop held at TIFR in December, 2011, giving an exposition of the Deligne-Lusztig theory. The aim of these lectures was 
to give an overview of the subject with several examples without burdening them with
detailed proofs of the theorems involved for which we refer to the original paper of Deligne 
and Lusztig, as well as the beautiful book of Digne and Michel on the subject.
The author thanks Shripad M. Garge for writing and texing the first draft of these notes. }
\end{quote}

\setcounter{tocdepth}{1}

\tableofcontents

{\sf
\vskip10mm
The Deligne-Lusztig theory constructs certain (virtual) representations of $G(\F_q)$ denoted as $R_G(T, \theta)$ 
where $G$ is a connected reductive algebraic group defined over $\F_q$, $T$ is a maximal torus in $G$ defined over $\F_q$ and $\theta$ is a character of $T(\F_q)$, $\theta:T(\F_q) \to \C^{\times} \stackrel{\sim}{\longrightarrow} \overline{\Q}_l^{\times}$. The representation 
$R_G(T, \theta)$ is called the Deligne-Lusztig induction of the character $\theta$ of $T(\F_q)$ to $G(\F_q)$, generalizing 
parabolic induction with which it coincides when  $T$ is a maximally split maximal torus in $G$.
It is known that if $T \subseteq M$ where $M$ is a Levi subgroup of a 
parabolic subgroup $P$  of $G$ defined over $\F_q$, 
then $R_G(T, \theta) = \Ind_{P}^G (R_{M}(T, \theta))$; thus $R_G(T, \theta)$ are really new objects only for those maximal tori which are not contained inside any Levi subgroup of 
a parabolic subgroup $P$  of $G$ defined over $\F_q$. 
These are exactly those tori which do not contain a non-central split torus of $\dim \geq 1$.

An irreducible representation $\pi$ of $G$ is said to be {\em cuspidal} if it does not appear as a factor of $\Ind_{P}^G \rho$ where $P$ is a proper parabolic and $\rho$ is an irreducible representation of $M$ considered as a representation of $P$. Any representation of $G(\F_q)$ can be understood in terms of cuspidal
representations of $G(\F_q)$, and of Levi subgroups of proper parabolic subgroups defined over $\F_q$. An important part of the subject is to study the {\em principal series representations}, $\Ind_{P}^G \rho$, where $\rho$ is an irreducible representation of $M$.
The decomposition of a  principal series representation into irreducible components is in general a difficult problem if one is interested in complete 
character information, but at least their parametrization can be understood.

\section{Preliminaries on  Algebraic groups}

Since Deligne-Lusztig theory constructs representations of $G(\F_q)$ using characters of all maximal tori in $G$ defined over $\F_q$, we begin with a review
of algebraic groups and their maximal tori.

An algebraic group over a field $k$, denoted by $G/k$, gives rise among other things to $G(A)$ which are groups for any algebra  $A/k$ 
associated in a functorial way, i.e., for every morphism of algebras $A_1 \to A_2$,  we get a morphism of groups $G(A_1) \to G(A_2)$.
An algebraic group $G$ is said to be {\em reductive} if $G \subseteq \GL(V)$, $V$ a finite dimensional 
vector space over $k$, and the representation $V$ of $G$ is completely reducible.
(In positive 
characteristic, we are not suggesting that {\it any} representation is completely reducible, but that there is one faithful
representation which is completely reducible.)
 
We say that $G$ is {\em semisimple} if $G$ is connected and there is no nontrivial homomorphism of algebraic groups $G \to A$ with $A$ abelian.

A good example of a reductive group to keep in mind --although too simple for the purposes of these notes-- is $G = \GL_n(\F_q)$ and $T = \prod_{i = 1}^d K_i^{\times}$ where $K_i$ are extensions of $\F_q$ of degree $n_i$ such that $\sum n_i = n$.
Among the maximal tori, there is the split torus, 
$$T = (\F_q^{\times})^n = \begin{pmatrix} \ast & & & \\ & \ast & & \\ & & \ddots & \\ & & & \ast \end{pmatrix} \subseteq \GL_n(\F_q)$$ 
and one anisotropic torus (mod center), $\F_{q^n}^{\times} \subseteq \GL_n(\F_q)$.
A Borel subgroup of $\GL_n(\F_q)$ is up to conjugacy the subgroup $B$ of the group of upper triangular matrices:
$$B = \begin{pmatrix} * & * & * & * \\ & * & * & * \\ & & \ddots & \vdots \\ & & & * \end{pmatrix},$$ 
and parabolic subgroups are (up to conjugacy) those subgroups of $G$ that contain this Borel subgroup.
It can be proved that subgroups of $G$ containing the group of upper triangular matrices are subgroups of the form 
$$P_{\alpha} = \begin{pmatrix} \GL_{d_1} & * & * & * \\ & \GL_{d_2} & * & * \\ & & \ddots & \vdots \\ & & & \GL_{d_k} \end{pmatrix}$$
where $\alpha  = (d_1, \dots, d_k)$ with $d_i \geq 1, \sum d_i = n$.
We have a decomposition, called the {\em Levi decomposition}, $P_{\alpha} = M_{\alpha} N_{\alpha}$ where $M_{\alpha} = \prod \GL_{d_i}(\F_q)$ is called a {\em Levi subgroup} of $P_{\alpha}$ and $N_{\alpha}$ is the unipotent radical of $P_{\alpha}$.
In fact, one has a split  exact sequence of groups $1 \to N_{\alpha} \to P_{\alpha} \to M_{\alpha} \to 1$.

\vskip3mm
Classification of reductive algebraic groups over $k$ is usually done in two steps:

\vskip3mm
\begin{itemize}
\item[{1.}] The theory for $k = \overline{k}$:
Here there are no surprises and the theory is the same as for compact, connected, Lie groups,  or  for $k = \C$.
Very briefly put, the theory says that any connected, reductive group over $k = \overline{k}$ is a product of groups of the type $A_n, B_n, C_n, D_n, G_2, F_4, E_6, E_7, E_8$ together with some center.
We have the classical groups
$$A_n \longleftrightarrow \SL_{n+1}, \hskip2mm B_n \longleftrightarrow \SO_{2n +1}, \hskip2mm C_n \longleftrightarrow \Sp_{2n}, \hskip2mm D_n \longleftrightarrow \SO_{2n}$$
and the remaining groups, $G_2, F_4, E_6, E_7, E_8$, are called exceptional groups.

\vskip3mm
\item[{2.}] Rationality questions:
The next step in the classification program is to understand what are $G'/k$ for a given $G/k$ such that $G/\overline{k} = G'/\overline{k}$.

This question is intimately related to the arithmetic of the field $k$.
This is studied among other cases in the following special cases:
\vskip2mm
\begin{itemize}
\item[{i)}] for $k = \R$: This was completed by E. Cartan more than a century back. For example, forms of $\SO(n)$ are $\SO(p, q)$ (Sylvester's law of inertia).

\vskip2mm
\item[{ii)}] for $k = \Q_p$: A complete understanding is known. 

\vskip2mm
\item[{iii)}] for $k = \F_q$: Here there are in most cases the only split groups, and in the other cases, one more which is quasi-split.
The groups of type $B_n = \SO_{2n+1}(\F_q), C_n = \Sp_{2n}(\F_q), G_2, F_4, E_7$ and $E_8$ have only the split form over $\F_q$ whereas for the groups of type $A_n, D_n$ and $E_6$ we have quasi-split form also:
$$A_n \longleftrightarrow \SL_{n+1}(\F_q), \SU_{n+1}(\F_q), \hskip5mm D_n \longleftrightarrow \SO_{2n}^{\pm}(\F_q),$$
where $+$ denotes the isometry group of a split quadratic form and $-$ denotes the quasi-split but non-split group which is the
isometry group of  the norm form of a quadratic extension added by $(2n-2)$-dimensional split quadratic space, while for $\SO_8$ we have one more form, because the Dynkin diagram of $\SO_8$ has an extra automorphism.
\end{itemize}
\end{itemize}

An important invariant associated to $G/\F_q$ is its order, $|G(\F_q)|$, e.g. $|\GL_n(\F_q)| = q^{n(n-1)/2} \prod_{i=1}^n(q^i - 1)$.
The Weyl group of $G$ is $W(G) := N(T)(\F_q)/T(\F_q)$ for a maximally  split maximal torus $T$ in $G$.
The Weyl groups for classical groups are given below. The Weyl group of $G_2$ is the dihedral group $D_6$ (of order 12) and the Weyl groups of other exceptional groups are slightly more difficult to describe. 
$$W(A_n) = S_{n+1}, \hskip5mm W(B_n) = W(C_n) = (\Z/2\Z)^n \rtimes S_n \hskip5mm {\sf ~and~} \hskip5mm W(D_n) = (\Z/2\Z)^{n-1} \rtimes S_n .$$
The Weyl group operates on $T$, and hence on the character group $X^*(T) = \Hom(T, G_m) \cong \Z^l$.
Therefore $W$ operates on the polynomial algebra $S^*(X^*(T)_{\R})$ on $X^*(T)_{\R} = X^*(T) \otimes_{\Z} \R$.
It is a theorem that $S^*(X^*(T)_{ \R})^W = \langle p_1, \dots, p_l \rangle$ where $p_i$ are homogeneous polynomials of degree $d_i$ called exponents of $G$.
These $d_i$ are:
\begin{enumerate}
\item[{}] $A_n: 2, 3, \dots, (n+1)$.
\item[{}] $B_n, C_n: 2, 4, 6, \dots, 2n$.
\item[{}] $D_n: 2, 4, \dots (2n-2), n$.
\item[{}] $G_2: 2, 6$.
\item[{}] $F_4: 2, 6, 8, 12$.
\item[{}] $E_6: 2, 5, 6, 8, 9, 12$.
\item[{}] $E_7: 2, 6, 8, 10, 12, 14, 18$.
\item[{}] $E_8: 2, 8, 12, 14, 18, 20, 24, 30$.
\end{enumerate}
A basic result about these exponents is that $\prod_{i=1}^l d_i 
= |W|$ and $\sum_{i=1}^l(d_i - 1) = N$ is the number of positive roots.

\begin{thm}
\begin{enumerate}
\item If $G$ is split then $|G(\F_q)| = q^N \prod_{i=1}^l (q^{d_i} - 1)$.
\item ${\displaystyle \lim_{q \to 1} \frac{|G(\F_q)|}{(q-1)^l} = \prod_{i=1}^l d_i =    |W|}$.
\end{enumerate}
\end{thm}

\begin{Prop}\label{Ennola}
$|\U_n(\F_q)| = |\GL_n(\F_{-q})|$.
\end{Prop}

\begin{rem}{\sf
One can  prove this proposition by elementary methods by directly calculating the 
order of the two groups involved. 
A proof via $l$-adic cohomology of the flag variety $G/B$ would be desirable which 
would then generalize to other quasi-split but non-split groups. (The recipe relating  the orders  of $\GL_n(\F_q)$ and $\U_n(\F_q)$ works also 
for outer form of $E_6(\F_q)$, and also for $D_n$, but only for $n$ odd.) 
It is part of 
what's called {\it Ennola duality} that one can go from the character table of $\GL_n(\F_q)$ 
to $\U_n(\F_q)$ by changing $q$ to $-q$.  
}\end{rem}

\section{Classical groups and tori}
Suppose $A$ is a semisimple algebra over a field $k$, i.e., an algebra which is isomorphic to a finite direct product of $\M_n(D)$ where $D$'s are  
division algebras over $k$. We note  that a semisimple algebra is canonically a product of simple algebras, in particular, a commutative semisimple algebra is canonically a product of field extensions of $k$.
Suppose $x \mapsto x^*$ is an involution on $A$, i.e., we have
$$(x+y)^*=x^* + y^*, \hskip5mm (xy)^* = y^* x^* \hskip5mm {\sf ~and~} \hskip5mm  (x^*)^* = x .$$
The involution $*$ preserves $k$ and may act as identity on $k$, in which case it is said to be of first type, or may act nontrivially on $k$, in which case 
it is  said to be of second type.
For an involution $*$ on a semisimple algebra $A$, we define $G(A, *) = \{x\in A: xx^* = 1\}$.

A classical group is a group of the form $G(A, *)$.
It is a theorem that any form of $A_n, B_n, C_n$ or $D_n$, except the triality forms of $D_4$ are groups of the form $G(A, *)$.

Let $C$ be a commutative semisimple subalgebra of $A$ of maximal dimension invariant under $*$. 

\begin{Lem}
The group $T(C) := \{x\in C: xx^* = 1\}$ is a maximal torus in $G(A, *)$, and any maximal torus of $G(A, *)$ is of this form.
\end{Lem}

\begin{proof}{\sf
We begin by noting that given a maximal torus $T$ in $G(A,*)$, the subalgebra of $A$ generated by a generic element of $T$, if $k$ is infinite, or 
in general the centralizer of $T$ in $A$, is a maximal commutative subalgebra of $A$ of the correct dimension (as one can prove assuming $k$ 
algebraically closed).  The converse follows by an explicit calculation assuming $k$ to be 
 algebraically closed.} \end{proof}

\begin{Lem} Let $(C, *)$ be any commutative semisimple algebra with involution over a field $k$ of dimension $n$ such that in the  decomposition
of $C$ as a direct sum of field extensions of $k$, there is at most one factor on which  the involution acts trivially, and that factor is $k$ (this forces $n$ to be odd). 
Fixing a unit  $d$ in $C$ with $d^*= \epsilon d$, $\epsilon = \pm 1$, 
we can construct a $\epsilon$-hermitian  form on $C$ by $\langle c_1,c_2 \rangle^d = \tr_{C/k}(c_1d c_2^*)$  with $T(C)$ a maximal torus of the corresponding unitary group.
\end{Lem}

\begin{proof} {\sf Note that  
any commutative algebra $C$ with involution can be embedded in the  semisimple algebra $A$,  
with $\dim C = \sqrt{\dim A}$, consisting of the algebra of $k$ endomorphisms of $C$ treated as a vector space over $k$. We use the involution on $C$ together with  the  element $d \in C$ 
with $d^* = \epsilon d$,  $\epsilon = \pm 1$,  to construct an involution to be denoted $*= *_d$ on $A$. 
The vector space $C$  comes equipped with a 
bilinear form given by $\langle c_1,c_2 \rangle_d = \tr_{C/k}(c_1 \cdot d \cdot c_2^*)$. 
This bilinear form is $\epsilon$-hermitian, i.e., $$\langle c_1,c_2 \rangle_d = \tr_{C/k}(c_1\cdot d \cdot  c_2^*) = \epsilon [\tr_{C/k}(c_2\cdot d \cdot  c_1^*)]^* =  \epsilon \langle c_2,c_1 \rangle^*_d .$$ 

If the involution on $C$ is of the first kind, fixing a nonzero element $d$ in $C$ with $d^*=-d$, 
we can construct a symplectic form on $C$ by $\langle c_1,c_2 \rangle_d = \tr_{C/k}(c_1d c_2^*)$.

The bilinear form $\langle c_1,c_2 \rangle_d$ on $C$  induces a natural involution $a\longrightarrow a^*$ on $A$ with the property that 
$\langle a c_1, c_2 \rangle_d = \langle c_1, a^* c_2 \rangle_d$. 
Under the right regular representation of $C$ on itself, $C$ embeds inside $A$ as a maximal 
commutative subalgebra, invariant under the involution on $A$ defined here since  $$\langle c \cdot c_1, c_2 \rangle_d =   \tr_{C/k}(cc_1d  c^*_2) = \tr_{C/k}(c_1d [c^*c_2]^*) = \langle c_1, c^* c_2 \rangle_d.$$ 
It is easy to check that $T(C) := \{x\in C: xx^* = 1\}$ is a maximal torus in  $G(A,*)$, and the converse follows from the previous Lemma.
}\end{proof}

\noindent{\bf Remark :} 
For groups $G = G(A,*)$ which have a unique $k$-rational form, in which case $A$ is unique too, 
such as for symplectic groups over any field, 
or  orthogonal groups in odd dimension, symplectic groups, and unitary groups over finite fields, 
this allows one to identify all maximal tori
of these classical groups. For embedding $T(C)$ in  a {\it given}
 orthogonal group in even number of variables over a finite field, there is an obstruction in terms of the discriminant of $(A, {}^*)$. {\it For general fields $k$, we are not able to say
which hermitian or quadratic forms  arise as }$\langle c_1,c_2 \rangle_d = \tr_{C/k}(c_1 \cdot d \cdot c_2^*)$ {\it as $d$ varies among the units of $C$ with $d^* = \pm d$, which  are presumably {\rm all} the hermitian or quadratic forms for which $T(C)$ is a maximal torus in the 
corresponding unitary group.} 

\vspace{4mm}

\noindent{\bf Remark :}

{\sf
\begin{enumerate}
\item Any maximal torus in $\GL_n(k)$ is of the form $\prod K_i^{\times}$ where $K_i/k$ are field extensions of $k$ with $\sum [K_i:k] = n$.
\item Any involution on $C =\prod K_i$ when acting on a particular $K_i$ takes it to a $K_j$ or preserves it, so we have $C  = \prod_I (K_i \times K_i) \times \prod_J K_j$.
For an extension $K/k$ which carries an involution $\iota$ with fixed field $K^\iota$, we denote by $K^1$ the kernel of the norm map $K^{\times} \to K^{\iota \times}$. 
Then any maximal torus in $G(A, *)$ is of the form $T(C) = \prod_I K_i^{\times} \times \prod_J K_j^1$.
\item For unitary groups one gets a commutative subalgebra $C$ left invariant under an  involution of the second kind. Define $C^{\circ} = \{x\in C:x^* = x\}$.
Then $k \otimes C^{\circ} \cong C$ as $*$-algebras.
This gives a bijection between tori in $\GL_n(k)$ (defined by $C^\circ$) and those in $\U_n(k)$ 
(defined by $C$) for $k$ a finite field.
\end{enumerate}
}

\section{Some Exceptional groups}
We saw above that a classical group can be described in terms of a semisimple, associative algebra. 
For exceptional groups, we need to consider non-associative algebras. This section plays no role in the rest of these notes, and are included just give some flavor of the groups involved in these cases.

Construction of exceptional groups begins with an {\it octonian} algebra $\Oc$ over $k$ which we first define.
These are  non-commutative, non-associative algebras but 
are endowed with an involution (i.e., an anti-automorphism) $x \mapsto \bar{x}$, such that $\tr(x): = x + \bar{x} \in k$, $\Nm(x): =x\bar{x} \in k$, which is a nondegenerate quadratic form
with $\Nm(xy) = \Nm(x)\Nm(y)$. 
One can realize the split octonian algebra $\Oc$ in the matrix form as follows:
$$\Oc = \begin{pmatrix} k & V \\ V^\vee & k \end{pmatrix} ,$$
where $V$ is a 3-dimensional vector space over $k$ with dual space $V^\vee$, 
together with a fixed 
isomorphism of $\Lambda^3 V $ with $k$, 
and hence also of $\Lambda^3V^\vee$ with $k$,
 with multiplication defined by,
$$\begin{pmatrix} a  &  v \\ w  & b  \end{pmatrix}  \begin{pmatrix} a' & v' \\ w' & b' \end{pmatrix} =  \begin{pmatrix} aa'+ \langle v, w' \rangle & av'+b'v - w \times w' \\ a'w+bw'+ v \times v'&  bb'+ \langle v', w \rangle  \end{pmatrix}, $$
where the products $\langle v, w' \rangle $ denotes the usual pairing 
between $V$ and $V^\vee$ and 
for vectors $v',v'' \in V$, $v' \times v'' $ is the vector in $V^\vee$ 
with the property that $\langle v, v' \times v''\rangle = 
v\wedge v' \wedge v'' \in k$ (through the fixed isomorphism of $\Lambda^3 V $ with $k$); 
similarly,  for vectors $w',w'' \in V^\vee$, $w' \times w'' \in V$. 

There are other, non-split, octonian algebras too defined in terms of 3-dimensional Hermitian vector spaces, but we will not get into that here.

The automorphism group of an octonion algebra is a group of type $G_2$. Thus $G_2$ comes equipped with a natural 7 dimensional representation over $k$ on trace zero elements of $\Oc$.

The descriptions of other exceptional groups involve exceptional Jordan algebras, $J = J(\Oc)$.
It is a 27 dimensional non-commutative, non-associative algebra over $k$.
An element of $J$ is of the following form:
$$A = \begin{pmatrix}a & z & \bar{y} \\ \bar{z} & b & x \\ y & \bar{x} & c \end{pmatrix}$$
where $a, b, c \in k$ and $x, y, z \in \Oc$.
We define a new product on $J$, $\alpha \circ \beta = \frac{1}{2}(\alpha\beta + \beta\alpha)$.
Then with this multiplication on $J$, the automorphism group of the algebra $J$ is the group $F_4$. Thus $F_4$ comes equipped with a natural 26
 dimensional representation over $k$ on trace zero elements of $J$.

We have a map $\det:J \to k$, obtained by taking the usual determinant of an element, which lies in $\Oc$, and then taking its trace from $\Oc$ to $k$, given by:
$$\det A = abc + \tr(xyz)- a\Nm(x)-b\Nm(y) -c\Nm(z).$$

The group $\Aut(J, \det)$ is the group $E_6$ which comes equipped with a natural 27 dimensional representation over $k$.

\section{Basic notions in representation theory}
Let $G$ be a reductive group, $P = MN$ be a proper parabolic in $G$.
Associated to a representation $\pi$ of $G(\F_q)$ there is a representation, called the {\em Jacquet module} $\pi_N$, of $M(\F_q)$ which is $\pi^{N(\F_q)}$, the subspace of $N(\F_q)$-fixed vectors in $\pi$.
Since $N(\F_q)$ is normalized by $M(\F_q)$, $\pi_N$ is a representation of $M(\F_q)$. Jacquet modules
play a very large role in representation theory.

A representation $\pi$ of $G(\F_q)$ is called {\em cuspidal} if $\pi_N = 0$ for unipotent radicals $N$ of every proper parabolic.

According to Harish-Chandra's philosophy of cusp forms, cuspidal representations are building blocks of all representations as the following theorem shows:

\begin{thm}
Let $\pi$ be an irreducible representation of $G(\F_q)$.
Then there exists a unique associated class of parabolics $P=MN$ and a unique cuspidal representation $\rho$ of $M(\F_q)$ such that 
$$\Hom_{G(\F_q)}(\pi, \Ind_{P(\F_q)}^{G(\F_q)} \rho) \ne 0 .$$
\end{thm}
Here uniqueness means that the pair $(M, \rho)$ is unique up to conjugacy by $G(\F_q)$.
Because of this theorem, the classification of $G(\F_q)$ splits into two parts:

\begin{enumerate}
\item Classify cuspidal representations of all Levi subgroups including $G$;
\item Understand the decomposition of $\Ind_{P(\F_q)}^{G(\F_q)} \rho$ for  $\rho$ a cuspidal representation of $M(\F_q)$. 
\end{enumerate}

Green answers both these parts in his paper for $\GL_n(\F_q)$.

For (1), cuspidal representations of $\GL_n(\F_q)$ are in one-to-one correspondence with regular characters $\chi:\F_{q^n}^{\times} \to \C^{\times}$, up to Galois action on $\F_{q^n}^{\times}$. 
A character $\chi$ of $\F_{q^n}^{\times}$ is called {\em regular} if $\chi^{\sigma} \ne \chi$ for all $1 \ne \sigma \in \Gal(\F_{q^n}/\F_q)$.

\begin{thm}\label{thm:Green}
For a regular character $\chi:\F_{q^n}^{\times} \to \C^{\times}$ there exists an irreducible cuspidal representation $\pi_\chi$ of $\GL_n(\F_q)$, 
such that for its character $\Theta_{\pi_{\chi}}$, we have:
\begin{enumerate}
\item[{i)}] $\Theta_{\pi_{\chi}}(su) = 0$ if $s \notin \F_{q^n}^{\times} \hookrightarrow \GL_n(\F_q)$.
\item[{ii)}] If $s$ generates a subfield $\F_{q^d} \subseteq \F_{q^n}$
$$\Theta_{\pi_{\chi}}(su) = (-1)^n (1-q^d)(1-q^{2d}) \cdots (1- q^{td}) \sum_{\sigma \in \Gal(\F_{q^d}/\F_q)} \chi(s^{\sigma})$$
where $t +1= \dim \ker (u-1)$ inside $\GL_m(\F_{q^d})$, the dimension being counted over $\F_{q^d}, n = md$.
\end{enumerate}
\end{thm}

To understand the decomposition of $\Ind_{P(\F_q)}^{G(\F_q)} \rho$ for  $\rho$ a cuspidal representation of $M(\F_q)$, and $G=\GL_n(\F_q)$, 
it suffices\footnote{\sf One observes that if we have distinct cuspidal representations 
 $\rho_i$, then the principal series representation $\rho_1 \times \cdots \times \rho_d$ is irreducible.} to understand the decomposition of $ \sigma \times \cdots \times \sigma$, a representation of $\GL_{dm}(\F_q)$ obtained by parabolic induction of the representation $\sigma \otimes \cdots \otimes \sigma$ of $\GL_m(\F_q)^d$ considered as a Levi subgroup of $\GL_{dm}(\F_q)$.
The problem gets reduced to understanding the case when $\sigma = 1, m = 1$.
At least there exists a bijective correspondence between irreducible complex representations of $\sigma \times \cdots \times \sigma$ and $1 \times \cdots \times 1$.
Finally $1 \times \cdots \times 1$ is nothing but $\Ind_B^G 1$.

\begin{thm}
$\End(\Ind_B^G 1, \Ind_B^G 1) \cong \C[W]$.
\end{thm}

\begin{cor}
Irreducible components of $\Ind_B^G 1$ are in one-to-one correspondence with $\widehat{W}$, the set of irreducible representations of $W$.
\end{cor}

Unipotent conjugacy classes in $\GL_n(\F_q)$ are in one-to-one correspondence with partitions of $n$ which also parametrize conjugacy classes in $S_n$ which is the Weyl group of $\GL_n(\F_q)$.

\section{The  Deligne-Lusztig theory}

Let us fix some notation. 
For any integer $n$, we denote by $n_p$, the $p$-primary component of $n$, 
 and $n_{p'} = n/n_p$. 
Define a  sign $\epsilon_G$ associated to a reductive group $G$ which plays a large role in character theory by 
$\epsilon_G:= (-1)^{r(G)}$, where $r(G)$ is the dimension of the maximal $\F_q$-split torus in $G$, the $\F_q$-rank of $G$. 

The work of Deligne and Lusztig was motivated by the following conjecture due to Macdonald.

\begin{conj}[{\sf Macdonald}]
To every maximal torus $T/\F_q$ contained in $G/\F_q$ and a character $\theta:T(\F_q) \to \C^{\times}$, which is regular in the sense that ${}^w\theta \ne \theta$ for all $w \in W:=N(T)(\F_q)/T(\F_q)$, $w \not = 1$,  there exists an irreducible representation $\pi_{\theta}$ of $G(\F_q)$ of dimension $|G(\F_q)|_{p'}/|T(\F_q)|$ 
whose value on a regular semisimple element $s \in G(\F_q)$ is nonzero only if $s \in T(\F_q)$ up to conjugacy. The character $\chi_{\theta}$ of $\pi_{\theta}$ at $s \in T(\F_q)$ is $\epsilon_T \epsilon_G \sum_{w \in W} \theta(s^w)$.
Further, the character $\chi_{\theta}$ at unipotent elements is independent of $\theta$.
\end{conj}

A part of the conjecture was that `most' of the representations of $G(\F_q)$ are irreducible $\pi_{\theta}$.

\vskip3mm
The work of Deligne-Lusztig proves Macdonald's conjecture and gives a geometric way of realizing them.

\begin{exam}{\sf
\begin{enumerate}
\item The dimension of a cuspidal representation of $\GL_n(\F_q)$ is $\prod_{i=1}^{n-1} (q^i - 1) = |\GL_n(\F_q)|_{p'}/|T(\F_q)|$ where $T(\F_q) = \F_{q^n}^{\times}$.
\item For $\GL_2(\F_q)$: $Ps(\alpha, \beta)(x, y) = (\alpha(x)\beta(y) + \alpha(y)\beta(x))$, where $\alpha, \beta$ are characters of $\F_q^\times$.
\item For $\GL_2(\F_q)$: $Ds(\chi)(x) = - (\chi(x) + \chi(x^{\sigma}))$, where $\chi$ is a 
regular character of $\F^\times_{q^2}$.
\end{enumerate}
}\end{exam}

Instead of defining $R_T^G(\theta)$ for any character $\theta$ of $T(\F_q)$, we define a more general representation $R_L^G(\rho)$ where $\rho$ is an 
irreducible representation of $L(\F_q)$, where $L$ is a subgroup of $G$ defined over $\F_q$, which is a Levi subgroup of a parabolic $P$ of $G$ defined over $\bar{\F}_q$, without 
requiring that $L$ is the Levi component of a parabolic defined over $\F_q$.

Begin with the algebraic variety $X_U=\{g \in G(\overline{\F}_q):g^{-1} F(g) \in U(\overline{\F}_q)\}$ where $P = LU$ is a parabolic in $G(\overline{\F}_q)$, and $F(g)$ denotes the image of 
$g\in G(\overline{\F}_q)$ under the Frobenius map.
The variety $X_U$ has left action of $G(\F_q)$, and right action of $L(\F_q)$.
The Deligne-Lusztig representation is on the $\overline{\Q}_l$-vector space 
$H_c^*(X)= \sum (-1)^i H^i_c(X_U, \overline{\Q}_l)$ admitting an action of $G(\F_q) \times L(\F_q)$. 
For an irreducible representation $\rho$ of $L(\F_q)$, the representation $R_L^G(\rho)$ is the $\rho$-isotypic component of this representation which is a (virtual) representation of $G(\F_q)$. As the 
notation suggests, $R_L^G(\rho)$ depends only on $L$ and not on the parabolic $P$ (defined over $\bar{\F}_q$) which is used to define $R_L^G(\rho)$.

We define a pairing $\langle -,- \rangle$ on the Grothendieck group of representations of a group by $\langle \pi, \pi' \rangle = \delta_{\pi, \pi'}$ for irreducible 
representations $\pi, \pi'$ and then extended  by linearity to the Grothendieck group.
One then has:

\begin{thm}
${\displaystyle \langle R^G_T(\theta), R^G_{T'}(\theta')\rangle = \frac{1}{ |T(\F_q)|}  \bigg[\big\{g \in G(\F_q): {}^gT = T', {}^g\theta = \theta'\big\} \bigg]}$
\end{thm}

Two pairs $(T, \theta)$ and $(T', \theta')$ are said to be {\it geometrically conjugate} if there exists $n \in \N$ such that $(T_n, \theta_n)$ and $(T'_n, \theta'_n)$ are conjugate. 
Here for  a maximal torus $T$ of $G$ defined over $\F_q$ and a character $\theta: T(\F_q) \to \C^{\times}$, the torus $T_n$ is the torus $T$ now considered in $G(\F_{q^n})$, and the character $\theta_n$ is the character of 
$T_n(\F_{q})$ given by  $T_n(\F_q) = T(\F_{q^n}) \stackrel{\Nm}{\longrightarrow} T(\F_q) \stackrel{\theta}{\longrightarrow} \C^{\times}$, where $\Nm:T(\F_{q^n}) \to T(\F_q)$ is the norm mapping.

\begin{thm}
The virtual representations $R^G_T(\theta)$ and $R^G_{T'}(\theta')$ are disjoint if and only if the pairs $(T, \theta)$ and $(T', \theta')$ are not geometrically conjugate. 
\end{thm}

The representation $R_L^G(\rho)$ is nothing but the parabolic induction if $L$ is a Levi subgroup of a parabolic $P$ which is defined over $\F_q$.
In that case, $X_U$ can be identified to $G(\F_q) \times_{U(\F_q)} U$.
The group $U$ being contractible, this space is for topological purposes essentially $G(\F_q)/U(\F_q)$ which is the home for the principal series representations.

The representations $R_L^G(\rho)$ have following nice properties:

1) Transitivity: If $L$ is a Levi subgroup in $G$ and $M$ is a Levi subgroup of $L$ then $R^G_L(R^L_M(\rho)) = R^G_M(\rho)$.

2) Behaviour under morphisms: If $\pi:G_1 \to G_2$ is an isomorphism up to centers, i.e., associated mapping from $G_1/Z_1 \to G_2/Z_2$ is an isomorphism of algebraic groups, then under such a morphism, tori in $G_1$ and $G_2$ correspond in a natural way (as the inverse image under $\pi$ of a maximal torus in $G_2$ is a maximal torus in $G_1$, and every maximal torus in $G_1$ arises in this way). 
We then have the following:

\begin{quote}
For a torus $T_2 \subseteq G_2$, let $\pi^{-1}(T_2) = T_1 \subseteq G_1$.
Thus we have a homomorphism of groups $T_1(\F_q) \to T_2(\F_q)$.
For $\theta_2:T_2(\F_q) \to \C^{\times}$,  composing with the homomorphism of groups $T_1(\F_q) \to T_2(\F_q)$, we get a character of $T_1(\F_q)$, which we denote as $\theta_2|_{T_1}$.
Then one has $R^{G_2}_{T_2}(\theta_2)|_{G_1} = R^{G_1}_{T_1}(\theta_2|_{T_1})$.
\end{quote}

We  have following corollaries of the two properties listed above: 

\begin{cor} 
\begin{enumerate}

\item If $T \subseteq L$ with $L$  a Levi subgroup of a parabolic subgroup of $G$ defined over $\F_q$, $R(T, \theta) : = R^G_T(\theta)$ has no cuspidal components, and therefore, the only tori $T$ for which $R(T,\theta)$ may 
have cuspidal representations are anisotropic mod center.

\item D-L representations of $\SL_n(\F_q)$ are nothing but restrictions of D-L representations of $\GL_n(\F_q)$.
\end{enumerate}
\end{cor}

\begin{Prop}
Suppose $\pi = \pi_{\chi}$ is an irreducible cuspidal representation of $\GL_n(\F_q)$ associated to $\chi:\F_{q^n}^{\times} \to \C^{\times}$.
Write the principal series representation $Ps(\pi, \pi)$ of $\GL_{2n}(\F_q)$ as $St_2(\pi) + Id_2(\pi)$ with 
$\dim St_2(\pi) \geq  \dim Id_2(\pi)$. Then $\dim St_2(\pi) = q^n \dim Id_2(\pi)$.
\end{Prop}

\begin{proof}{\sf
We will prove below that $St_2(\pi) - Id_2(\pi)$ as a virtual character is the same as $R^{\GL_{2n}(\F_q)}_{\F_{q^{2n}}^{\times}}(\chi\circ \Nm)$, where $\chi \circ \Nm$ is the character of 
${\F^\times_{q^{2n}}} $ given by:
$ \F_{q^{2n}} \stackrel{\Nm}{\longrightarrow} \F_{q^n} \stackrel{\chi}{\longrightarrow} \C^{\times}$. In particular,
$$\dim [St_2(\pi)-Id_2(\pi)] = (q^{2n-1}-1)\cdots (q-1) .$$

Since,
$$\dim Ps(\pi, \pi) = \frac{|\GL_{2n}(\F_q)|_{p'}}{|T(\F_q)|} = \frac{(q^{2n}-1)\cdots (q-1)}{(q^n -1)^2},$$
we find that,

$$\dim St_2(\pi)= (q-1)\cdots (q^{2n-1}-1)\left[\frac{(q^{2n}-1)}{(q^n-1)} + 1\right] = \frac{2q^n}{(q^n-1)}$$
and 
$$\dim Id_2(\pi)= (q-1)\cdots (q^{2n-1}-1)\left[\frac{(q^{2n}-1)}{(q^n-1)} - 1\right] = \frac{2}{(q^n-1)} .$$
In particular,  $\dim St_2(\pi)/\dim Id_2(\pi) = q^n$.
}\end{proof}

We now prove the observation that was used here: $St_2(\pi) - Id_2(\pi)$ as a virtual representation 
is the same as $R^{\GL_{2n}(\F_q)}_{\F_{q^{2n}}^{\times}}(\chi\circ \Nm)$.

\begin{proof}{\sf
Recall that $R^G_{T_1}(\theta_1)$ and $R^G_{T_2}(\theta_2)$ are orthogonal unless $T_1 = T_2$, $\theta_1 = \theta_2$, up to conjugacy by $G(\F_q)$, and disjoint if and only if geometrically distinct. 
It follows that the representations $R^{\GL_{2n}(\F_q)}_{\F_{q^n}^{\times} \times \F_{q^n}^{\times}} (\chi \times \chi)$ and $R^{\GL_{2n}(\F_q)}_{\F_{q^{2n}}^{\times}}(\chi \circ \Nm)$ 
are orthogonal but have some components in common 
because $(T_1,\theta_1)$ and $(T_2,\theta_2)$ are geometrically conjugate as we check below.
By the orthogonality relations for the Deligne-Lusztig representations, 
both are sums of two distinct representations. 
The first one of them is $St_2(\pi) + Id_2(\pi)$ so the only choice for the latter one is $St_2(\pi) - Id_2(\pi)$.

Now we check that $(T_1,\theta_1)$ and $(T_2,\theta_2)$ are geometrically conjugate.
Both the pairs, $(\F^\times_{q^n} \times \F^\times_{q^n}, \chi \times \chi)$ and 
$(\F_{q^{2n}}^{\times}, \chi \circ \Nm)$ give,   after the base change to $\F_{q^{2n}}$, the same 
character 
$(\chi, \chi^{\sigma}, \dots, \chi^{\sigma^{n-1}}, \chi, \chi^{\sigma}, \dots, \chi^{\sigma^{n-1}})$ of 
$[(\F_{q^n} \oplus \F_{q^n}) \otimes_{\F_q} \F_{q^{2n}}]^\times 
= [\oplus_{i = 1}^n \F_{q^{2n}} + \oplus_{i = 1}^n \F_{q^{2n}}]^\times$, and 
of $[\F_{q^{2n}} \otimes_{\F_q} \F_{q^{2n}} = \oplus_{i=1}^{2n} \F_{q^{2n}}]^\times $.
Therefore,
$(T_1, \theta_1), (T_2, \theta_2)$ are geometrically conjugate.
}\end{proof}

\section{Unipotent cuspidal representations}

{\em Unipotent representations} are constituents of $R_T^G 1$ for  maximal tori $T$ in $G/\F_q$.
These are the building blocks of all the representations of $G(\F_q)$.
Irreducible constituents of $\Ind_B^{G(\F_q)} 1$ are examples of unipotent representations.
These are parametrized by representations of relative  Weyl group $W$ of $G$ (which we recall 
is defined to be $N(T)(\F_q)/T(\F_q)$ for $T$ a maximal torus of $G$ which is maximally 
split). 
More precisely, $\Ind_B^{G(\F_q)} 1 = \sum_{\alpha \in\widehat{W}} d(\alpha) \pi_{\alpha}$, where $d(\alpha)$ is the dimension of the irreducible representation $\alpha$ of $W$.

Cuspidal unipotent representations are unipotent representations which are also cuspidal.

\begin{thm}
For $G = \GL_n(\F_q)$ there are no cuspidal unipotent representations except for $n =1$.

For $G = \U_n(\F_q)$, there is a cuspidal unipotent representation if and only if $n$ is a triangular number, i.e., $n = 1+ 2+ \cdots + m$ for some $m$.
\end{thm}

For classical groups, these representations exist once in a while, whereas for exceptional groups there are plenty of them.

\begin{exam}[{\sf Construction of cuspidal unipotent representation for $\U_3(\F_q)$}]{\sf
If $\U_3$ is defined by $z_1\bar{z}_1 + z_2 \bar{z}_2 + z_3\bar{z}_3$, let $T = K^1 \times K^1 \times K^1$, acting 
on each $z_i$ by $t_iz_i$.
Then $R_T^{\U_3} 1 
= -1 - 2 [(q^2 -q)] + St_3$.
The middle term,  of dimension $q^2-q$, is (twice) a cuspidal unipotent representation as we check now. 

To analyze  $R_T^{\U_3} 1 $, note that by the inner product formula, we have 
$\langle R_T^{\U_3} 1, R_T^{\U_3} 1\rangle = |W| = 6$.

Let $T_0= K \times K^1$ be the maximal torus contained in a Borel subgroup of $\U_3(\F_q)$.
Then one  has $\langle R^{\U_3}_T 1, R^{\U_3}_{T_0} 1\rangle = 0$, where $R^{\U_3}_{T_0} = 1 + St_3$.
By generalities, $R^{\U_3}_T 1$ always contains $1$ and $St$. 
Thus, by these relations, the only options left for $R_T^{\U_3} 1$ is $\pm(1 - St_3 + 2\pi)$.

We now prove  that $\pi$ is cuspidal. 
By Proposition \ref{propn:aniso} below, the Jacquet module of $R_T^{\U_3} 1$ (with respect to a Borel subgroup) is zero, but the Jacquet modules of $1$ and $St_3$ are the same so the Jacquet module of $\pi$ must be zero.

We finally fix the sign in $R_T^{\U_3} 1 = \pm(1 - St_3 + 2\pi)$.
For this, one computes the dimension 
$$\dim R^{\U_3}_T = \frac{|\U_3|_{p'}}{(q+1)^3} = (q+1)(q^2-1)(q^3+1)/(q+1)^3 = q^3-2(q^2-q)-1,$$
so the sign must be negative, i.e., $R_T^{\U_3} 1 = -(1 - St_3 + 2\pi)$.
}\end{exam}

\begin{propn}\label{propn:aniso}
The virtual representation $R_T^G \theta$ is cuspidal in the sense that all its Jacquet modules are zero if and only if $T$ is anisotropic. 
\end{propn}

\begin{proof}{\sf
One part follows from the transitivity of Deligne-Lusztig induction: if $T$ is isotropic, then $T \subset M$, for a 
proper Levi subgroup $M$ of $G$, in which case, $R_G(T,\theta) = R_G(M, R_M(T, \theta)),$ and therefore no irreducible
component of $R_G(T,\theta)$ can be cuspidal.

For the other part, note that $\langle R^G_T \theta, R^G_P \rho \rangle = 0$ if $T$ is anisotropic.
}\end{proof}

\section{The Steinberg representation}
The {\em Steinberg representation} is an irreducible representation of $G(\F_q)$ of $\dim = q^n$ where $n = \dim U$ where $U$ is the unipotent radical of a Borel subgroup of $G$. It plays a ubiquitous role 
in all of representation theory. It is defined as follows:

Fix a Borel subgroup $B_0 \subseteq G$, then $St_G = \sum (-1)^{l(P)} \Ind_{P \supseteq B_0}^G 1$ where $l(P)$ is the number of simple roots in the Levi of the parabolic $P$.
Equivalently, this is the space of functions on $ B_0(\F_q) \backslash G(\F_q)$
 modulo functions invariant on the right by a bigger parabolic.

For any rank 1 group, such as $\GL_2(\F_q)$ or $\U_3(\F_q)$, the Steinberg representation is realized on the space of functions on $ B_0(\F_q) \backslash G(\F_q)$  modulo constants, because $G$ is the only parabolic containing $B_0$ in this case.

Steinberg representation has many interesting properties.
For instance, $\Theta_{St}(g) = 0$ if $g$ is not semisimple and for semisimple elements $s \in G(\F_q)$, $\Theta_{St}(s) = \epsilon_G \epsilon_{Z_G^{\circ}(s)}|Z_G^{\circ}(s)|_p$.
In particular, it is non-zero on semisimple elements, and 
 for regular semisimple elements this character is $\pm 1$.
Further, the Jacquet module of the Steinberg representation of $G$ is the  Steinberg 
representation of the Levi subgroup.

\section{Dimension of the Deligne-Lusztig representation}
In this section as a sample  computation
 we  prove that, 
$$ \dim R(T, \theta) = \epsilon_G \epsilon_T \frac{|G(\F_q)|_{p'}}{|T(\F_q)|},$$ 
a statement which requires several ingredients to prove.

The main geometric input in the proof of this dimension formula will be the Lefschetz fixed point formula. 
Recall that for any variety over a field $k$, there are
the $\ell$-adic cohomology groups with compact support $H^i_c(X, \Q_\ell)$ which are finite dimensional
vector spaces over $\Q_\ell$, and non-zero only for $i \in [0,2 \dim X]$. Any automorphism
$\phi$ of $X$ operates on these cohomology groups, and the alternating sum of their
traces is called the Lefschetz number of $\phi$, and denoted by ${\mathcal L}(\phi, X)$.

\begin{thm} Let $\phi$ be an automorphism of finite order $n = n_p\cdot n_{p'}$ on a variety
$X$ over a field of characteristic $p$. Decompose $\phi=\phi_s\phi_u$ as a product of automorphisms
$\phi_s=\phi^{a\cdot n_p}$ of order coprime to $p$, $\phi_u=\phi^{b\cdot n_{p'}}$ of order a power of  $p$ where $a,b$ are integers with $an_p+bn_{p'}=1$. Then 
we have the equality of the Lefschetz numbers,
$${\mathcal L}(\phi, X) = {\mathcal L}(\phi_u,X^s),$$
where $X^s$ is the sub-variety of $X$ on which $\phi_s$ acts trivially; 
in particular, if the order of $\phi$ is coprime of $p$, then, 
$${\mathcal L}(\phi, X) = {\mathcal L}(Id,X^\phi).$$
\end{thm}

If we are in the context of fixed point free action, then $X^s$ is empty, giving the following
corollary, which is what we would use later.

\begin{cor} Let $\phi$ be an automorphism of finite order prime to $p$  of  a variety
$X$ over a field of characteristic $p$ which operates without fixed points on $X$.
Then ${\mathcal L}(\phi, X) = 0.$ 
\end{cor}

We next define  a certain notion 
of duality on the Grothendieck group of representations of $G(\F_q)$ of finite length which 
has been independently observed by many people (in different contexts!), to which the names of 
Alvis, Aubert, Curtis, Deligne, Lusztig, Schneider, Stuhler, Zelevinsky is associated.  
(Twisting by the sign character seems like a good  analogue of this for Weyl groups, in particular for symmetric groups.)

\begin{thm}
Let $D_G(\pi) = \sum_{P \supseteq B_0} (-1)^{l(P)} \Ind_P^G(\pi_N)$.
Then
\begin{enumerate}
\item $D_G^2 = Id$.
\item $D_G$ is an isometry on the Grothendieck group of representations 
of $G(\F_q)$, and as a consequence,  $D_G$ takes irreducible representations to irreducible 
representations up to a sign.
\item $D_G$ commutes with parabolic induction, and also with Deligne-Lusztig induction.
\item $D_G$ flips 1 and $St_G$ 
\end{enumerate}
\end{thm}

On cuspidal representations,  $D_G$ is identity, up to a sign.

\vskip3mm
Note that $\langle R_T^G 1, 1\rangle =\epsilon_G \epsilon_T$ and hence by applying $D_G$ we get that $\langle R^G_T 1, St\rangle = \epsilon_G\epsilon_T$. We have  $\langle R^G_T \theta, St\rangle = 0$ if $\theta \ne 1$ since by 
the theorem about {\it geometric conjugacy}, the virtual representations $R_T^G 1$ and $R_T^G\theta$ have no 
irreducible representations in common.

\vskip3mm
Therefore $\langle H_c^*(X_U), St\rangle 
= \langle \sum_{\theta} R^G_T \theta, St\rangle = \epsilon_G \epsilon_T$. On the other hand, since the character of the Steinberg representation is supported
only on semisimple elements, therefore only on elements of order coprime to $p$, the inner
product $\langle H_c^*(X_U), St\rangle $ is calculated by a sum only over  elements of order coprime to $p$, 
on which, unless the element is the trivial element, the character of $H_c^*(X_U)$ is zero by corollary 8.2
(the action of $G(\F_q)$ on $X_U$ being given by translation is fixed point free).

Therefore,  $\sum_{s \, \,{\sf semisimple}} \Theta_{H^*_c(X_U)}(s) \Theta_{St}(s^{-1}) = |T(\F_q)| \dim R^G_T(\theta) q^{\dim U}$.
This allows us to calculate $\dim R^G_T \theta$ as desired.

\section{The Jordan decomposition of representations}

We have defined the notion of a unipotent representation of $G(\F_q)$. There is also the notion of a semisimple
representation which as a first approximation is defined to be those irreducible representations of $G(\F_q)$ 
whose dimension is coprime to $p$. (This a good definition only for ``good'' primes for $G$; in general, one defines 
an irreducible representation to be semisimple if the average of its character values on regular unipotent 
elements is nonzero, in which case the average is known to $\pm 1$.)

Lusztig has constructed a Jordan decomposition for  representations of $G(\F_q)$ which gives a bijection between the set of
irreducible representations $\pi$ of $G(\F_q)$, and pairs $(\pi_s, \pi_u)$ of irreducible representations in which
$\pi_s$ is a semisimple irreducible representation of $G(\F_q)$, and $\pi_u$ is a unipotent representation of   $Z_{\widehat{G}}(s_\pi)$,  
where $s_\pi$ is a semisimple element associated to $\pi$ in the dual group $\widehat{G}$.  The Jordan decomposition
 has the key property that,
$$\dim \pi = \dim \pi_s \cdot \dim \pi_u.$$

We discuss some of this in a little more detail beginning 
with the notion of geometric conjugacy of the pair $(T,\theta)$ which 
has a nice interpretation in terms of the Langlands dual group $\widehat{G}$ associated to the reductive group $G$.
We will not go into details of the dual group here, but emphasize that we take $\widehat{G}$ to be defined over the same field as $G$. For $G=\GL_n(\F_q), \U_n(\F_q), \SO_{2n}(\F_q),$ the group $\widehat{G}$ 
can be taken to be $G$ itself, and for $G = \Sp_{2n}(\F_q)$, $G^\vee = \SO_{2n+1}(\F_q)$.

\begin{thm}
\begin{enumerate}
\item To each pair $(T, \theta)$, there corresponds a well-defined  semisimple conjugacy class $[s]$ in $\widehat{G}(\F_q)$.
\item Two pairs $(T_1, \theta_1)$ and $(T_2, \theta_2)$ are geometrically conjugate if and only if the associated elements $s_1$ and $s_2$ are conjugate in $\widehat{G}(\overline{\F}_q)$.
\item The correspondence $(T, \theta) \mapsto [s]$ gives a bijective correspondence between geometric conjugacy classes of pairs $(T, \theta)$ and semisimple conjugacy classes in $\widehat{G}(\F_q)$ up to conjugacy in $\widehat{G}(\overline{\F}_q)$ which is also called geometric conjugacy in $\widehat{G}(\F_q)$.
\end{enumerate}

\end{thm}

\vskip3mm
For a semisimple element  $s \in \widehat{G}(\F_q)$, define  the {\it Lusztig series} ${\mathcal E}(G, s)$ to be 
 the union of constituents of $R^G_T \theta$ such that $(T, \theta)$ corresponds to  an element in the geometric conjugacy class defined by $s$. As $s$ varies in 
$ \widehat{G}(\F_q)$ up to geometric conjugacy,  the Lusztig series ${\mathcal E}(G, s)$ gives
rise a disjoint partition of the set of all irreducible representations of $G(\F_q)$. 
 An important example of the 
Lusztig series is ${\mathcal E}(G,1)$ which is the set of all 
irreducible unipotent representations of $G(\F_q)$. 

\vskip3mm
The following theorem is called the {\it Jordan decomposition of representations}. 
Before we state the theorem, note that for  a semisimple element $s \in \widehat{G}(\F_q)$, the centralizer 
of $s$ in $\widehat{G}(\bar{\F}_q)$, denoted by  $Z_{\widehat{G}}(s)$,  
is a reductive group defined over
$\F_q$. It is known that if the center of $G$ is connected, such as for an adjoint group, then 
$Z_{\widehat{G}}(s)$ 
is a connected reductive group. We will assume this to be the case in the next theorem, although Deligne-Lusztig theory has an extension to disconnected reductive groups too.
\begin{thm}
There exists a natural bijection between ${\mathcal E}(G, s)$ 
and ${\mathcal E}(Z_{\widehat{G}}(s), 1)$, $\rho \mapsto \chi_{\rho}$  which when extended linearly to Grothendieck group 
of representations takes the Deligne-Lusztig representation $R_G(T,\theta)$ to $R_{Z_{\widehat{G}}(s)}(\widehat{T},1)$ up to a sign, and has the 
property  that
$$ \dim \rho = \frac{|G|_{p'}}{|Z_{\widehat{G}}^{\circ}(s)|_{p'}} \dim \chi_{\rho}.$$
\end{thm}

\noindent{\bf Example :}
Observe that since there is no difference between conjugacy and geometric conjugacy 
in $\GL_d(\F_q)$, although in general the Lusztig series 
${\mathcal E}(G, s)$ is a collection of irreducible representations appearing in several $R(T,\theta)$, in the case of $\GL_d(\F_q)$, ${\mathcal E}(G, s) = R(T,\theta)$,
and in fact, it is the set of irreducible components of a principal series representation of $\GL_d(\F_q)$ induced from a cuspidal representation
of a Levi subgroup.

We now use the Jordan decomposition to work out the dimensions of irreducible representations of $\GL_{mn}(\F_q)$ appearing in the
parabolically induced representation $\pi \times \cdots \times \pi$ for $\pi$ a cuspidal representation of $\GL_m(\F_q)$.  

In this case, $\widehat{G}(\F_q) = \GL_{mn}(\F_q)$, and we take $s = t \times \cdots \times t$ to be a block diagonal matrix with $t$ a 
regular elliptic element in $\GL_m(\F_q)$ represented by a regular element of $\F^\times_{q^m}$. The centralizer of $s$ in $\GL_{mn}(\F_q)$ 
is $\GL_n(\F_{q^m}) \hookrightarrow \GL_{mn}(\F_q)$. Thus 
the irreducible representations appearing in the  principal series representation $\pi \times \cdots \times \pi$ 
for $\pi$ a cuspidal representation of $\GL_m(\F_q)$, are parametrized by the unipotent
representations of $\GL_n(\F_{q^m})$, 
which are parametrized by the partitions of $n$; further, the  
the ratios of the dimensions of irreducible representations appearing in the  principal series representation $\pi \times \cdots \times \pi$ 
is the same as the ratios of the dimensions of irreducible representations appearing in the  principal series representation $\GL_n(\F_{q^m})$
induced from the trivial character of the Borel; this generalizes the conclusion of Proposition 5.1.

\section{Exercises}

\subsection{Algebraic groups}

\begin{enumerate}
\item If $K/k$ is a degree $d$ extension of a field $k$ then $K^{\times} \subseteq \GL_d(k)$.
\item  Prove that $|\U_n(\F_q)| = |\GL_n(\F_{-q})|$.
\item For a torus $T/\F_q$, we have $|T(\F_q)| = 1$ if and only if $q = 2$ and $T$ is split.
\item Reconcile the classification of classical groups discussed in section 2 in terms of 
algebras with involutions  with what you know about forms of $\GL_n, \U_n, \SO_n, \Sp_n$.
\item For $\prod K_i^{\times} \hookrightarrow \GL(\oplus K_i)$ calculate the Weyl group. 
\item For the tori constructed inside $T \hookrightarrow G(A, {}^*) \hookrightarrow \Aut(A)$, calculate the Weyl group of $T$ inside $G(A, *)$ and compare the answer with the Weyl group of $T$ inside $\Aut(A)$.
\item (Root system) Let $\U_n(k)$ be defined by the hermitian form 
$$\begin{pmatrix} & & & 1 \\& & 1 & \\ & \ddots & & & \\ 1 & & & \end{pmatrix} ,$$
i.e., $H(\overrightarrow{z}, \overrightarrow{w}) = \sum z_i \bar{w}_{n+1-i}$.
Let
$$T = \begin{pmatrix} t_1 & & & & & \\ & \ddots & & & & \\ & & t_m & & & \\ & & & t_m^{-1} & & \\ & & & & \ddots & \\ & & & & & t_1^{-1} \end{pmatrix}, \hskip3mm T = \begin{pmatrix} t_1 & & & & & & \\ & \ddots & & & & & \\ & & t_m & & & & \\ & & & 1 & & & \\ & & & & t_m^{-1} & & \\ & & & & & \ddots & \\ & & & & & & t_1^{-1} \end{pmatrix}$$
be the respective maximal split tori for $n=2m$ and for $n = 2m+1$ with $t_i \in k^\times$.
Calculate the root space decomposition.
\item Identify the maximal tori in $\Sp_4(\F_q)$ up to conjugacy. Are there isomorphic but non-conjugate tori? Are there isomorphic but non-stably conjugate tori? (Two tori $T_1$ and $T_2$ 
are said to be stably conjugate in $G(k)$  of there exists $g \in G(\bar{k})$ taking $T_1$ to $T_2$ by inner-conjugation which as a homomorphism from  $T_1$ to $T_2$ 
is defined over $k$. 
 
\end{enumerate}

\subsection{Representation theory}

\begin{enumerate}
\item Classify representations of $\GL_2(\F_q)$ which do not remain irreducible when restricted to $\SL_2(\F_q)$. How many are there?
\item The principal series $Ps(\chi_1, \chi_2)$ of $\GL_2(\F_q)$ is a sum of 2 irreducibles when restricted to $\SL_2$ if and only if $\chi_1\chi_2^{-1}$ is a nontrivial character of order 2.
\item Classify those principal series representations of $\SL_2(\F_q)$ which are reducible. 
\item Classify irreducible representations of $\GL_2(\F_q)$ which have a self-twist, i.e., $\pi = \pi \otimes \chi$, $\chi\ne 1$.
\item Identify the unique cuspidal representation of $\GL_2(\F_2) \cong S_3$.
\item Identify the degrees of the cuspidal representations of $\SL_2(\F_4) \cong A_5$.
\item Identify distinct conjugacy classes in $\SL_2(\F_q)$ which become the same in $\GL_2(\F_q)$.
\item Identify  the affine curve $XY^q-X^qY=1$ over $\F_q$ to be the Deligne-Lusztig variety 
$X_U$ defined in these notes associated to a non-split torus in $\SL_2(\F_q)$.
\item Identify the automorphism group of the curve $XY^q-X^qY=1$ over $\F_q$.
\item In Theorem \ref{thm:Green} prove that $u \in \GL_m(\F_{q^d})$, 
\item Calculate the dimensions of the two components of $\rho \times \rho$ for $\rho$ cuspidal of $\GL_2(\F_q)$.
\item Given a unipotent element $u_{\lambda}$ in $\GL_n$ for a partition $\lambda$ of $n$, calculate the number of Borel subgroups  such that $u_{\lambda} \in B$.
\item Calculate: $\Ind_B^{\GL_3(\F_q)}1$ as a sum of irreducibles and verify that we have one irreducible constituent of dimension $1$, two of dimension $(q^2 - q)$ and one of dimension $q^3$, the Steinberg.
\item Do the above exercise for $\U_3(\F_q)$.
\end{enumerate}

\subsection{Further Exercises}

\begin{enumerate}
\item Prove that there exists a natural bijective correspondence between representations of 
$\GL_n(\F_q)$ over $\overline{\Q}_l^{\times}$, 
 and conjugacy classes in $\GL_n(\F_q)$ 
after having made the choice of an embedding 
$\overline{\F}_q^{\times} \hookrightarrow \overline{\Q}_l^{\times}$.

\item Prove that there is an isomorphism: 
${\displaystyle \lim_{\longrightarrow} X^*(\F_{q^n}^{\times}) \cong  \overline{\F}_q^{\times}}$ as Gal$({\overline{\F}}_q/\F_q)$-modules where the groups $X^*(\F_{q^n})$ form a directed system of groups via the norm mappings: $\F^\times_{q^d} \rightarrow 
\F^\times_{q^{d'}}$ for $d'|d$.(Hint: For any cyclic group $C$ of order $d$, the 
automorphism group of $C$ is canonically isomorphic to $(\Z/d)^\times$.)
\end{enumerate}

\vskip3cm

\section{References}

\vskip 1cm
P. Deligne and G. Lusztig: Representations of reductive groups over finite fields, {\it Annals of Mathematics, 103} (1976), 103-161.

\vskip 5mm
F. Digne and J. Michel: Representations of finite groups of Lie type, {\it 
London Math Society, Student text 21}, (1991).

\vskip 5mm

J.A. Green: The characters of the finite general linear groups, {\it Transactions of the AMS}, 80 (1955)402-447.

\vskip 1cm

}\end{document}